\documentclass[11pt]{amsart}
\usepackage{amsmath,amsthm,latexsym,amssymb}
\usepackage{enumerate}
\usepackage{esint}
\usepackage{color}
\numberwithin{equation}{section}

\newtheorem{theorem}{Theorem}
\newtheorem{lemma}{Lemma}

\newtheorem{proposition}{Proposition}

\numberwithin{theorem}{section}

\numberwithin{corollary}{section}
\numberwithin{lemma}{section}
\numberwithin{definition}{section}
\numberwithin{proposition}{section}
\numberwithin{remark}{section}

\newcommand{\rn}{\mathbb R^n}
\newcommand{\R}{\mathbb R}

\def\Sn{\mathcal{S}_n}

\newcommand{\medint}{-\kern  -,375cm\int}

\begin{document}
\title[\textsf{The classical overdetermined Serrin problem}]{The classical overdetermined Serrin problem }
\author[  C. Nitsch  - C. Trombetti]{  C. Nitsch$^*$   - C. Trombetti$^*$}
\thanks{%
$^*$ Dipartimento di Matematica e Applicazioni ``R. Caccioppoli'', Universit\`{a}
degli Studi di Napoli ``Federico II'', Complesso Monte S. Angelo, via Cintia
- 80126 Napoli, Italy; email: c.nitsch@unina.it;
cristina@unina.it}
\date{}

\begin{abstract}
In this survey we consider the classical overdetermined problem which was studied by Serrin in 1971. The original proof relies on Alexandrov's moving plane method, maximum principles, and a refinement of Hopf's boundary point Lemma. Since then other approaches to the same problem have been devised. Among them we consider the one due to Weinberger which strikes for the elementary arguments used and became very popular. Then we discuss also a duality approach involving harmonic functions, a shape derivative approach and a purely integral approach, all of them not relying on maximum principle. For each one we consider pros and cons as well as some generalizations.
\end{abstract}
\maketitle

\tableofcontents

\section{The early years}
In a celebrated paper \cite{Serrin} Serrin initiates the study of elliptic equations under overdetermined boundary conditions. He establishes in particular the radial symmetry of
the solution to the following overdetermined Poisson problem.
Let $\Omega$ be a {  bounded, } smooth, open, connected set of {  $\rn$} , and {  let} $\nu_x$ the outward  normal {  at $x \in \partial \Omega$}, if $u$ is a smooth solution to
\begin{equation}\label{delta}
\left\{
\begin{array}{ll}
\Delta u=-1 & \mathrm{in}\,\,\Omega \\
u=0 & \mathrm{on}\,\,\partial \Omega \\
\displaystyle\frac{\partial u}{\partial \nu_x}=const(=c) &  \mathrm{on}\,\,\partial \Omega,
\end{array}
\right. 
\end{equation}
then $u=\frac{R^2-|x|^2}{2n}$ up to a translation and therefore $\Omega$ is a ball with radius $R$.
The main tool of his proof is a technique introduced  by Alexandrov \cite{A1,A2} known as \emph {moving plane} (in a completely different context to established that {  the only compact, embedded $(n-1)$-dimensional smooth hypersurfaces in $\rn$ with constant mean curvature are the spheres}) combined with a clever refinement of the maximum principle (see Lemma \ref{maxpri}). Right after Serrin's paper, Weinberger \cite{Wein} came out with a very short proof of the same result using the maximum principle applied to an auxiliary function. However in spite of its simplicity Weinberger's proof on one hand seems to rely on the linearity of the Laplace operator and was not elementarily generalizable to nonlinear ones, on the other hand is restricted to constant righthand side in the Poisson problem. Serrin's proof has in fact the great advantage of being easily stretchable to a wide range of fully nonlinear elliptic operators with fairly general data.\\
Shortly after, these early papers resulted in a wide research field which nowadays it is still very prominent. It is also very important to mention that Serrin's approach inspired other { fundamental }results concerning symmetry in PDE's. Among the others a seminal paper by Gidas, Ni and Nirenberg \cite{GNN} which unfortunately is beyond the scope of the present survey. Our goal indeed is to summarize in a concise but self contained way both Serrin's and Weinberger's proofs along with some of the alternative results which came out more recently. Nonlinear problems, stability issues, possibility to extend the symmetry result in case of lack of regularity, overdetermined problems in exterior domains, different overdetermined boundary conditions, are only few of the interests which became popular during the last few decades. It is impossible to give an exhaustive list of all the results hence we will restrict our attention just on the original Poisson problem \eqref{delta}. We shall stress pros and cons of every approach and mention major applications to different settings.

As Serrin explains \cite{Serrin}, his work originated form physical motivations:
\begin{quote} \it Consider a viscous incompressible fluid moving in straight parallel streamlines
through a straight pipe of given cross sectional form $\Omega$. If we fix rectangular coordinates
in space with the $z$ axis directed along the pipe, it is well known that the
flow velocity $u$ is then a function of $x$, $y$ alone satisfying the Poisson differential
equation (for $n = 2$)
$$\Delta u=-A \quad\mbox{in $\Omega$}$$
where $A$ is a constant related to the viscosity and density of the fluid and to the rate
of change of pressure per unit length along the pipe. Supplementary to the differential
equation one has the adherence condition
$$u=0 \quad\mbox{on $\partial\Omega$}.$$
Finally, the tangential stress per unit area on the pipe wall is given by the quantity
$\mu\frac{\partial u}{\partial \nu_x}$ where $\mu$ is the viscosity. 
Our result states that the tangential stress on the pipe wall is the same at all points of the wall 
if and only if the pipe has a circular cross section.

Exactly the same differential equation and boundary condition arise in the
linear theory of torsion of a solid straight bar of cross section $\Omega$, (...) 
when a solid straight bar is subject to torsion, the magnitude
of the resulting traction which occurs at the surface of the bar is independent
of position if and only if the bar has a circular cross section.
\end{quote}

In order to understand why the boundary overdetermination is so interesting in physical context one has to notice that {   it may arise in optimal control theory}. Following for instance the analogy with the torsion problem, we can ask what is the shape of a prismatic bar that maximizes the torsional rigidity when the cross sectional area is assigned. This is the famous Saint-Venant problem and the answer is the provided by the bar of circular cross section \cite{PSZ}. A necessary condition that a smooth cross section $\Omega$ has to satisfy, for being  the bar a maximizer of the torsion, is stationarity among smooth domain variations. The Torsion becomes a so called shape functional and the problem is recast in the framework of the shape optimization via domain derivative \cite{HenPi}. {  As mentioned by Serrin the torsion problem consists in finding a function $u$ (called torsion function) which solves $$\Delta u=-A \quad\mbox{in $\Omega$}$$ and $$u=0 \quad\mbox{on $\partial\Omega$}.$$} The well known Hadamard formula for the torsional rigidity enforces the gradient of $u$ to be constant on the boundary of $\Omega$ (i.e. the bar has constant shear stress) and here comes the overdetermination. In view of Serrin's result we can state that, when optimizing the torsion of a prismatic bar with respect to area preserving smooth variations of the cross section, the circular shaft is the unique stationary point.

\subsection{Serrin's {  result} \cite{Serrin}}

{  Before stating the result we observe that every solution to \eqref{delta}  is positive in $\Omega$ and that divergence theorem together with the fact that $Du=c\nu$ on  $\partial \Omega$ give 

\[
c= -\dfrac{|\Omega|}{|\partial \Omega|}.
\]
}
The main Theorem reads as follows.
\begin{theorem}\label{S1}
Whenever $\Omega$ is a $C^2$ bounded domain (bounded open and connected) of
  $\R^n$ and $u\in C^2(\overline\Omega)$ is a solution to problem \eqref{delta} then, up to a translation, $u=\frac{R^2-|x|^2}{2n}$ and $\Omega$ is a ball
with radius $R$.
\end{theorem}
 
We start by recalling Serrin's proof of Theorem \ref{S1}, which  relies on the moving planes method together with the strong maximum principle.
\begin{proof}
We denote by $H_\nu$  an open halfspace with unit outer normal $\nu$   and we move  this halfspace  along the direction $\nu$ until it intersects $\Omega$. 
We still denote by $H_\nu$ the halfspace after its motion and  by $\Omega_\nu= \Omega \cap H_\nu$.
For every cap $\Omega_\nu$ let  us denote by $\Sigma(\Omega_\nu)$  its reflection with respect to $\partial H_\nu$ and let us move $H_\nu$ until $\Sigma(\Omega_\nu)\subset \Omega$. When the motion stops then one of these two {  cases} occur:
\medskip
\begin{enumerate}
\item $\Sigma(\Omega_\nu)$  becomes internally tangent to $\partial \Omega$ at a point $\bar x$ not belonging to $\partial H_\nu$;
\\
\item $H_\nu$ reaches a position such that $\nu$ is tangent to $\partial \Omega$ at some point   $\bar y$ .
\end{enumerate}

We denote by $H'_\nu$ the halfspace when it reaches one of these positions and by $\Omega'_\nu$  the respective cap. 
The goal is to prove that $\Omega$ is symmetric with respect to the hyperplane $\partial H_\nu$. Once this fact is proved then the theorem follows, since   for every direction $\nu$, $\Omega$ would be symmetric with respect to the hyperplane normal to $\nu$. Moreover, by construction, $\Omega$ would also be simply connected,   then it has to be a ball and the unique solution to \eqref{delta} is the paraboloid.

Let now $\Sigma(\cdot)$ denote the reflection across $\partial H'_\nu$.
{  We set $v$ the function defined in $\Sigma(\Omega'_\nu)$ by
\[
v(x)= u(\Sigma(x)) \quad x \in \Sigma(\Omega'_\nu).
\]} 
Obviously $v$ satisfies:

\begin{equation}\label{reflection}
\left\{
\begin{array}{ll}
\Delta v=-1 & \mathrm{in}\,\,\Sigma(\Omega'_\nu) 
\\ \\
v=u & \mathrm{on}\,\,\partial \Sigma(\Omega'_\nu) \cap \partial H'_\nu
 \\ \\
v=0\quad  \mathrm{ and}\quad \displaystyle\frac{\partial v}{\partial \nu_x}=c &  \mathrm{on}\,\,\partial \Sigma(\Omega'_\nu)  \setminus \partial H'_\nu.
\end{array}
\right. 
\end{equation}

Since $\Sigma(\Omega'_\nu)$ is contained in $\Omega$ one can consider the function $u-v$ and (recalling that $u>0$ in $\Omega$) observe that it satisfies

\begin{equation}\label{difference}
\left\{
\begin{array}{ll}
\Delta (u-v)=0 & \mathrm{in}\,\,\Sigma(\Omega'_\nu) 
\\ \\
u-v =0 & \mathrm{on}\,\,\partial \Sigma(\Omega'_\nu) \cap \partial H'_\nu
 \\ \\
u-v \ge 0,   &  \mathrm{on}\,\,\partial \Sigma(\Omega'_\nu)  \setminus \partial H'_\nu.
\end{array}
\right. \notag
\end{equation}

At this point the strong maximum principle {  gives} either
\begin{equation}
\label{positive}
u-v>0 \quad \mathrm{in}\,\, \Sigma(\Omega'_\nu) 
\end{equation}
or $u \equiv v$ in $ \Sigma(\Omega'_\nu)$. The latter case would imply that $\Omega$ is symmetric about $\partial H'_\nu$.

Assume that {  case} (1) occurs, that is   $\Sigma(\Omega'_\nu)$  is internally tangent to $\partial \Omega$ at a point $\bar x$ not belonging to $\partial H_\nu$ and assume by contradiction that \eqref{positive} holds true. Then Hopf Lemma ensures that
\[
\displaystyle\frac{\partial }{\partial \nu_{\bar x}}(u-v) >0,
\]
but this contradicts the fact that \eqref{delta} and in \eqref{reflection} yield
\[
\displaystyle\frac{\partial u }{\partial \nu_{\bar x}}=\displaystyle\frac{\partial v}{\partial \nu_{\bar x}}  =c.
\]
We conclude that \eqref{positive} cannot occur in case (1).

Case (2) is much more { complicated }since Hopf Lemma cannot apply. The proof makes use of  a refinement of the maximum principle, see Lemma  \ref{maxpri} below (for its proof see \cite{Serrin}). 
The goal is to prove that $u-v$ has in $\bar y$ a second order zero. To do this we fix a coordinate system with the origin at $\bar y$, the $x_n$ axis in the direction of the inward normal to $\partial \Omega$ at $\bar y$ (that is $-\nu_{\bar y}$), and the $x_1$ axis in the direction of $\nu$, that is normal to $\partial H'_\nu$.
In this coordinates system the boundary of $\Omega$  is locally given by

\[
x_n= \phi(x_1,x_2, \cdots,x_{n-1}) \quad \phi \in C^2.
\]
Since $u \in C^2$ the boundary conditions, {  $u=0$ on $\partial \Omega$ and $\displaystyle\frac{\partial u }{\partial \nu_{x}}=c$ on $\partial \Omega$,} can be written as 
\begin{equation}
\label{bc1}
u(x_1,x_2, \cdots,x_{n-1},\phi)=0,
\end{equation}
and
\begin{equation}
\label{bc2}
\displaystyle\frac{\partial u }{\partial x_n} - \displaystyle\sum_{k=1}^{n-1} \displaystyle\frac{\partial u }{\partial x_k}\displaystyle\frac{\partial \phi }{\partial x_k}=
c \left \{  1+ \sum_{k=1}^{n-1} \left(\displaystyle\frac{\partial \phi }{\partial x_k}\right)^2\right \}^{1/2},
\end{equation}
respectively.

Differentiating \eqref{bc1} with respect to $x_i$, for {  $ i=1,\dots, n-1$,} we have
\begin{equation}
\label{diff1}
\displaystyle\frac{\partial u }{\partial x_i}+\displaystyle\frac{\partial u }{\partial x_n}\displaystyle\frac{\partial \phi }{\partial x_i}=0.
\end{equation}
Evaluating \eqref{diff1} and \eqref{bc2} at $\bar y$ and recalling that $\displaystyle\frac{\partial \phi }{\partial x_i}(\bar y) = 0$ for { $i=1, \dots, n-1$} we have
\begin{equation}
\label{gradzero}
\displaystyle\frac{\partial u }{\partial x_i}(\bar y)=0 \quad \displaystyle\frac{\partial u }{\partial x_n}(\bar y)=c.
\end{equation}
Differentiating \eqref{diff1}   with respect to $x_j$, we get for { $i,j=1, \dots, n-1$}

\begin{equation}
\label{diff2}
\displaystyle\frac{\partial^2 u }{\partial x_i\partial x_j}(\bar y) +c \displaystyle\frac{\partial^2 \phi }{\partial x_i\partial x_j}(\bar y)=0
\end{equation}
while differentiating \eqref{bc2} with respect to $x_i$, for $i=1, \cdots n-1$ and using \eqref{gradzero} we obtain 

\begin{equation}
\label{diffn}
\displaystyle\frac{\partial^2 u }{\partial x_n\partial x_i}(\bar y) =0.
\end{equation}
From \eqref{diff2} 

\begin{equation}
\label{diffnn}
\displaystyle\frac{\partial^2 u }{\partial x_n^2}(\bar y) = -\displaystyle\sum_{i=1}^{n-1} \displaystyle\frac{\partial^2 u }{\partial x_i^2}(\bar y) -1= c \Delta \phi(\bar y) -1.
\end{equation}
{ By construction} $\Sigma(\Omega'_\nu) \subseteq \Omega$ { and} all the second derivatives $\displaystyle\frac{\partial^2 \phi }{\partial x_1\partial x_j}(\bar y)=0$ for
$j=2, \cdots, n-1$, { because} $\displaystyle\frac{\partial \phi }{\partial x_1}$ has an extremum point at $\bar y$ with respect to all but the first coordinates directions.
Setting
\[
v(x_1,x_2,\cdots,x_n) = u(-x_1,x_2,\cdots,x_n),
\]
by \eqref{diff2}, \eqref{diffn} and the last remark we have that all the first and second derivatives of $u$ and $v$ coincide at $\bar y$.
The function $w= u-v$ satisfies
 { \[
\Delta w= 0 \quad {\rm in } \>\> \Sigma(\Omega'_\nu),
\]}
{ \[
w>0  \quad {\rm in } \>\> \Sigma(\Omega'_\nu),
\]}
and $w(\bar y)=0$. If $\theta $ is any direction not parallel to $\nu$ Lemma \ref{maxpri} ensures that {  either}
\[
\displaystyle\frac{\partial (u-v) }{\partial \theta}(\bar y)>0  \quad {\rm or } \>\> \displaystyle\frac{\partial^2 (u-v) }{\partial \theta^2}(\bar y)>0,
\]
which is a contradiction since all the first and second derivatives of $u$ and $v$ coincide at $\bar y$.
\end{proof}

\noindent

The following Lemma is a refinement of Hopf Lemma. We omit its proof which is contained in \cite{Serrin}
\begin{lemma}
\label{maxpri}
Let $\Omega$ be a $C^2$ bounded domain (bounded open and connected) of
  $\R^n$ and let $\nu$ a direction such that $<\nu,\nu_y>=0$, $y \in \partial \Omega$. Let $H_\nu$  be an open halfspace with unit outer normal $\nu$,  $\Omega_\nu= \Omega \cap H_\nu$ and let  $w \in C^2(\bar \Omega_\nu)$ satisfy
  $$
  \Delta w \le 0 \quad {\rm in} \>\> \Omega_\nu,
  $$ 
   $w\ge 0$ in $\Omega_\nu$ and $w(y)=0$. If $\theta$ is a direction in $y$ entering $\Omega_\nu$ such that $<\theta,\nu_y>\neq 0$, then either
  \[
  \frac{\partial w}{\partial \theta}(y) >0 \quad {\rm or} \> \>  \frac{\partial^2 w}{\partial \theta^2}(y) >0
    \]
    unless $w \equiv 0$.
\end{lemma}
\subsubsection{Remark on the proof and generalization.}
The great advantage of Serrin's proof with respect to all other techniques that we are going to analyze is that it works out of the box on a massive number of other problems.
The main ingredients used are:
\begin{itemize}
\item The problem is invariant under reflection
\item In any boundary point of $\Omega$, in a framework where one of the axis points into the normal direction, the second derivative of $u$ can be determined in terms of the other second order derivative.
\item Maximum principle and boundary point maximum principle hold.
\end{itemize} 
If for instance we consider 
\begin{equation*}
\left\{
\begin{array}{ll}
\Delta u=f(u,|Du|) & \mathrm{in}\,\,\Omega \\
u=0 & \mathrm{on}\,\,\partial \Omega \\
\displaystyle\frac{\partial u}{\partial \nu_x}=const(=c) &  \mathrm{on}\,\,\partial \Omega.
\end{array}
\right. 
\end{equation*}
then $\Omega$ is a ball and $u$ is radially symmetric, provided $f$ is differentiable and $u>0$. 
The condition $u>0$ {  is unavoidable in order to apply the moving plane.} 
The eigenvalue problem serves as a counterexample. 
No symmetry of solutions can be established via moving planes for
\begin{equation}\label{Schiffer}
\left\{
\begin{array}{ll}
\Delta u=-\lambda u & \mathrm{in}\,\,\Omega \\
u=0 & \mathrm{on}\,\,\partial \Omega \\
\displaystyle\frac{\partial u}{\partial \nu_x}=const(=c) &  \mathrm{on}\,\,\partial \Omega.
\end{array}
\right. 
\end{equation}
unless we know that we are dealing with the first eigenvalue $\lambda$ {  where $u$ has constant sign}.
The radial symmetry of solutions to the overdetermined eigenvalue problem \eqref{Schiffer} is known as Schiffer conjecture.


Serrin moving plane technique can be generalized also to many nonlinear elliptic operators (such as $p$-Laplacian), but the effectiveness of the proof depends upon the fine structure of the equation, and it is not possible to give an exhaustive list of the nonlinearity covered. 

Finally we notice that it is possible to consider also different boundary conditions. For instance replace the constant $c$ in \eqref{delta} with a smooth monotone non decreasing function of the mean curvature of $\partial\Omega$.


\subsection{Weinberger's proof \cite{Wein}}\label{subswein}

The proof makes use of an integral identity (Poho\v{z}aev identity), and of the strong maximum principle applied to an auxiliary function called $P$-function.

We recall the Poho\v{z}aev identity 
\begin{proposition}
Let  $g\in C^1(\R)$ be a nonnegative function and let $G(u)=\int_u^0 g(s)ds$. If $u \in C^2(\Omega)\cap C^1(\overline{\Omega})$ is a solution to the problem
$$\left\{\begin{array}{ll}
\Delta u=g(u)&\mathrm{in } \>\Omega\\
u=0&\mathrm{on }\> \partial\Omega
\end{array}
\right.$$
in a smooth domain $\Omega$ of $\rn$, 
 then
\begin{equation} \label{pohozaev}
\frac{n-2}{2}\int_\Omega |Du|^2 \>dx +\frac{1}{2}\oint_{\partial\Omega}<x,\nu_x>|Du|^2= n \int_\Omega G(u)\>dx.
\end{equation}
\end{proposition}
\begin{proof}
For the proof see for instance \cite{Struwe}, \cite{Poh65}.
\end{proof}

\begin{proof}[Proof of Theorem  \ref{S1}]

Let $u$ be a solution to \eqref{delta}, by \eqref{pohozaev}  we have 
 \[
 \frac{n-2}{2}\int_\Omega |Du|^2 \>dx + \frac{c^2}{2} \oint_{\partial \Omega} <x, \nu_x> = n \int_\Omega u \>dx.
 \]
 Equation in \eqref{delta} and the divergence theorem give:
\begin{equation}
\label{uguale}
\int_\Omega |Du|^2 \> dx=\int_\Omega u \>dx \quad { \rm and}\quad \oint_{\partial \Omega} <x, \nu_x>= n |\Omega|.
\end{equation}
{ Therefore we get}
\begin{equation}
\label{mean}
(n+2) \int_\Omega u \> dx = n c^2 |\Omega|.
\end{equation}
The classical Schwarz's inequality and equation in \eqref{delta} give 
\begin{equation}
\label{schwarz}
1= (\Delta u)^2 \leq n \sum_{i=1}^{n} \left(\frac{\partial^2 u}{\partial x_i^2}\right)^2 \leq n  \sum_{i,j=1}^{n} \left(\frac{\partial^2 u}{\partial x_i \partial x_j}\right)^2,
\end{equation}
so the function $P= |Du|^2 + \frac{2}{n}u$ satisfies

\begin{equation}
\label{pfunction}
\Delta \left(|Du|^2 + \frac{2}{n}u\right) = 2 \sum_{i,j=1}^{n} \left(\frac{\partial^2 u}{\partial x_i \partial x_j}\right)^2 - \frac{2}{n} \ge 0.
\end{equation}
From the strong maximum principle, since $|Du|^2 + \frac{2}{n}u = c^2$ on $\partial \Omega$, we conclude that either

\[
|Du|^2 + \frac{2}{n}u < c^2 \quad  {\rm in} \quad \Omega
\]
or

\[
|Du|^2 + \frac{2}{n}u \equiv c^2 \quad  {\rm in} \quad \Omega.
\]
{  In the first case by  \eqref{uguale} we have}
\[
\frac{n+2}{n} \int_\Omega u \> dx < c^2 |\Omega|,
\]
which contradicts  \eqref{mean}. { Therefore $P$ is constant in $\Omega$. This implies equality in both \eqref{pfunction} and \eqref{schwarz}, 
and we deduce that
\[
\frac{\partial^2 u}{\partial x_i \partial x_j}= \frac{\delta_{ij}}{n}.
\]
Consequently 
\[
u=\frac{R^2-|x|^2}{2n}
\]
up to translations and $\Omega$ is a ball of radius $R$.}
\end{proof}
\subsubsection{Remark on the proof and generalization{ s} .}
Weinberger's proof is particularly attractive for its elementary arguments. With respect to Serrin's proof it requires less regularity. Indeed already only interior maximum principle for the auxiliary P-function and the Poho\v{z}aev identity are needed. For this reason $u\in C^2(\Omega)\cap C^1(\bar\Omega)$ is enough. This also means that Weinberger broadens the class of domain among which the symmetry result can be established. Moreover Garofalo and Lewis showed in \cite{GL} that is is possible to assemble via $P$-function a Weinberger argument also for $p$-Laplacian type operators and recast the problem in the Sobolev $W^{1,p}$ settings. This paper opened new perspectives on a technique which for many years have been prescribed to the linear case. Operator in divergence form of $p$-Laplacian type have been later considered for instance in \cite{FaKa,FGK}, and even the special case of the $\infty$-Laplacian has been handled in \cite{BuKaw,CraFra}.

\section{More recent alternative proofs }
\subsection{The duality Theorem \cite{PaSc}}
The duality Theorem shows a deep connection between Serrin's overdetermined problem and the mean value theorem for harmonic functions. It is well known that the average of an harmonic function in a ball always equals the average on its boundary. Serrin's result established that the mean value theorem can be true only on balls in the sense that, if the average on a smooth, {  bounded} domain $\Omega$ equals the one on $\partial \Omega$ regardless the harmonic function we consider, then $\Omega$ must be a ball. In what follows $\Omega$ is a smooth domain.
\begin{theorem}
\label{D1}
{  Let $u \in C^2(\Omega) \cap C^1(\bar \Omega)$ be the solution to $-\Delta u=1$ in $\Omega$, and $u=0$ on $\partial \Omega$. The following statements are equivalent:}
\begin{enumerate}[(i)]
\item  { $u$ is a solution to \eqref{delta}.}
\item ${  \dfrac1{|\Omega|}\displaystyle\int_\Omega h \> dx = \dfrac1{|\partial\Omega|} \displaystyle\oint_{\partial \Omega} h } \quad{ \rm for \>all \> functions} \>h \in {  C^{0}({\bar \Omega})} \>{ \rm harmonic \>in\>} \Omega.$
\end{enumerate}
\end{theorem}

\begin{proof}
Assume (i). Divergence theorem immediately implies that for every $h$ harmonic in $\Omega$
\[
\int_\Omega h \>dx = \int_\Omega (-\Delta u) h \> dx = \displaystyle\oint_{\partial \Omega} - \frac{\partial u}{\partial \nu} \,h = { \dfrac{|\Omega|}{|\partial \Omega|} }\displaystyle\oint_{\partial \Omega}  \,h
\]
and then (ii).
Conversely assume (ii) and let { $u \in C^2(\Omega) \cap C^1(\bar \Omega)$ }be such that $-\Delta u=1$ in $\Omega$, $u=0$ on $\partial \Omega$. Then
\[
0=\int_\Omega h \>dx -{ \dfrac{|\Omega|}{|\partial \Omega|} } \displaystyle\oint_{\partial \Omega} h \> dx= { \displaystyle\oint_{\partial \Omega} - (\frac{\partial u}{\partial \nu} +  \dfrac{|\Omega|}{|\partial \Omega|} )h}  
\]
which implies $\displaystyle\frac{\partial u}{\partial \nu}= -\dfrac{|\Omega|}{|\partial \Omega|} $  choosing $ { \displaystyle h=\frac{\partial u}{\partial \nu} +  \dfrac{|\Omega|}{|\partial \Omega|}}$ on $\partial\Omega$.

\end{proof}

\begin{theorem}
\label{D2}
If (ii)  holds true then $\Omega$ is a ball.
\end{theorem}

\begin{proof}
{  By Theorem \ref{D1} there exists $u \in C^2(\Omega) \cap C^1(\bar \Omega)$ solution to \eqref{delta}. Let $h= <x,Du> -2u$. The function $h$ is harmonic in $\Omega$ (observe that it is harmonic in the distributional sense thanks to the differential identity 
$\Delta(<x,Dv>)= <D(\Delta v),x> + 2\Delta v$ and by classical regularity results $h$ is smooth in $\Omega$).} { (ii) together with $Du = c \nu$ on  $\partial \Omega$} give
\[
\int_\Omega <x,Du>\> dx -2 \int_\Omega u \> dx  = -c^2\oint_{\partial\Omega} <x, \nu_x> =-c^2n |\Omega|.
\]
{  In view of the divergence theorem} 
\begin{equation}
\label{equality33}
c^2n |\Omega| -(n+2) \int_\Omega u \> dx  =0.
\end{equation}
Let $P= |Du|^2 +\frac{2}{n}u$, since
\[
 \int_\Omega u (\Delta P) \> dx = -\int_\Omega P\> dx -c\oint_{\partial\Omega }P,
\]
{  the fact that $u$ is a solution to \eqref{delta}  together with  \eqref{equality33} yields}
\[
\int_\Omega u (\Delta P) \> dx = -\left(\frac{n+2}{n}\right) \int_\Omega u \> dx  -c^3 |\partial \Omega|=  - \left(\frac{n+2}{n}\right) \int_\Omega u \> dx + c^2|\Omega|=0.
\]
Here we have used that $c = -\dfrac {|\Omega|}{|\partial\Omega|}$.

\noindent
The strong maximum principle {  leads} $u>0$ in $\Omega$, then \eqref{pfunction} implies that $\Delta P=0$ in $\Omega$ and the proof concludes as in Subsection \ref{subswein}.
\end{proof}

\subsubsection{Remark on the proof and generalization.}
To our knowledge the proof by duality theorem is the first one which does not make explicit use of maximum principle. We face a flavor of the maximum principle when $u$ is assumed {  to have constant sign}. Nevertheless the proof is reminiscent of Weinberger's one and indeed both share the same regularity of $u$. Finally there is also an interesting generalization due to Bennett \cite{Be} where fourth order overdetermined problem for biharmonic operator are taken into account.

\subsection{The domain derivative \cite{ChoHen}}

As we have seen in the first section, there is a deep connection between overdetermination and shape optimization. Throughout this section if $\Omega$ is an open subset of $\R^n$ satisfying the assumptions of Theorem \ref{S1} such that there exists a solution to problem \eqref{delta}, we say that $\Omega$ is a solution to Serrin's problem. Following \cite{ChoHen} we are going to show how to construct a shape functional which is minimized by solutions to Serrin's Problem and then infer the uniqueness of the minimizer.

For every $\omega \subset \R^n$ with $C^2$ boundary we denote by $u_\omega$ the solution to

\begin{equation}\label{deltino}
\left\{
\begin{array}{ll}
\Delta u_\omega=-1 & \mathrm{in}\,\,\omega \\
u_\omega \in H_0^1(\omega) \\
\end{array}
\right. 
\end{equation}
and we consider the functional 
\[
J(\omega)= n \oint_{\partial \omega} |Du_\omega|^3 -(n+2) \int_\omega |Du_\omega|^2 \> dx.
\]

The strategy of the proof consists in proving that every solution to Serrin's problem minimizes the functional $J$. 
Indeed the following lemma holds true

\begin{lemma}
\label{funct}
$J(\omega) \ge 0$ for every $\omega \subset \R^n$ with $C^2$ boundary. If $\omega$ is a solution to Serrin's problem then $J(\omega)=0$.
\end{lemma}
\begin{proof}

Multiplying  \eqref{schwarz} by $u_\omega$ ($u_\omega>0$ in $\omega$) and recalling that
$$\Delta(|Du_\omega|^2) = 2 \sum_{i,j=1}^{i=n} \left(\frac{\partial^2 u_\omega}{\partial x_i \partial x_j}\right)^2 $$ we have

\begin{equation}
\label{integral1}
\int_\omega u_\omega \>dx \leq \frac{n}{2} \int_\omega u_{\omega}\Delta(|Du_\omega|^2)  \> dx.
\end{equation}
{ The divergence} theorem, the fact that $u_\omega=0$ and the fact that $\frac{\partial u_\omega}{\partial \nu_x}= -|Du_\omega|$ on $\partial \omega$ {  bring}

\begin{equation}
\label{integral2}
\int_\omega u_\omega \>dx \leq \frac{n}{2} \left[\oint_{\partial\omega} |Du_{\omega}|^3 + \int_\omega |Du_{\omega}|^2 (\Delta u_\omega)\> dx \right].
\end{equation}
{  Equation} in \eqref{deltino}, {  together with}
\begin{equation}
\label{aster}
\int_\omega u_\omega \>dx= \int_\omega |Du_{\omega}|^2 \> dx 
\end{equation}
{  carry}
\begin{equation}
\label{integrale3}
0 \leq \frac{n}{2} \oint_{\partial\omega} |Du_{\omega}|^3 - \left(\frac{n}{2}+1\right)\int_\omega |Du_{\omega}|^2 \> dx,
\end{equation}
{  which proves the inequality}.

Now assume that $\omega$ is a solution to Serrin's problem, then $c= \frac{\partial u_\omega}{\partial \nu_x}= -|Du_\omega|$ on $\partial \omega$.

From \eqref{mean} and \eqref{aster} we have that
\[
(n+2) \int_\omega |Du_{\omega}|^2 \> dx = nc^2 |\omega|,
\]
{ therefore}
\[
J(\omega) = -nc^3 |\partial \omega| -nc^2|\omega|.
\]
{  Since} $c =- \frac{|\omega|}{|\partial \omega|},$ the thesis follows.
\end{proof}

We briefly recall the definition of shape derivative and Hadamard formula (we refer for instance to \cite{had, HenPi}){ .}

Let $\omega$ be a smooth open set in $\R^n$, and let $\theta \in C^2(\R^n;\R^n)$  and denote by $\omega_t= \{ x+t \theta(x), x \in \omega  \}, \>\> t>0$. The derivative of $J$ at $\omega$ in the direction $\theta$ is 

\begin{equation}
\label{shapederiv}
dJ(\omega, \theta) = \lim_{t \rightarrow 0^+} \frac{J(\omega_t)-J(\omega)}{t}.
\end{equation}
The computation of \eqref{shapederiv} leads to calculate also the derivative of $u_\omega$ with respect to the domain. Such a derivative denoted by $u_\omega'$ satisfies
\begin{equation}\label{deltino2}
\left\{
\begin{array}{ll}
\Delta u'_\omega=0 & \mathrm{in}\,\,\omega \\
u'_\omega = -\displaystyle\frac{\partial u}{\partial \nu_x} <\theta,\nu_x> &  \mathrm{on}\,\,\partial \omega.
\end{array}
\right. 
\end{equation}

\begin{lemma}
\label{sd}
The derivative of the functional $J$ at $\omega$ in the direction $\theta$ is given by
\begin{equation}
\label{derivJ}
dJ(\omega,\theta) = \oint_{\partial\omega} \left (\left[(2n-2)|Du_\omega|^2 -2n(n-1)H|Du_\omega|^3    \right]   <\theta,\nu_x> -3n |Du_\omega|^2  \displaystyle\frac{\partial u'_\omega}{\partial \nu_x}\right)
\end{equation}
where $H$ is the mean curvature of $\partial \omega$ and $u'_\omega$ is defined in \eqref{deltino2}.
\end{lemma}

\begin{proof}
The proof follows {  from} Hadamard formula (see  \cite{had,HenPi}).
{  For the function $j_1(\omega)= \displaystyle\int_\omega f(\omega) \>dx$ such a formula reads
\[
dj_1(\omega,\theta)= \int_\omega f'(\omega)\> dx + \oint_{\partial \omega} f(\omega) <\theta,\nu_x> ,
\]
while for $j_2(\omega) = \displaystyle\oint_{\partial \omega} g(\omega) $
we have
\[
dj_2(\omega,\theta)= \oint_{\partial \omega} g'(\omega)\> + \oint_{\partial \omega} (n-1)H \,g(\omega) <\theta,\nu_x> + \oint_{\partial \omega} \frac{\partial g(\omega)}{\partial \nu_x} \, <\theta,\nu_x>.
\]
Here $f'(\omega)$ and $g'(\omega)$ denote the derivatives with respect to the domain of $f$ and $g$, respectively.} 
From these two formulae applied to $J$ we get

\begin{equation}\label{der1}
\begin{array}{ll}
dJ(\omega,\theta)= & 3n  \displaystyle\oint_{\partial \omega}|Du_\omega| <Du_\omega,Du'_\omega> + n(n-1) \oint_{\partial \omega} |Du_\omega|^3  H <\theta,\nu_x>  \\ \\
&+ n \displaystyle\oint_{\partial \omega} \frac{\partial |Du_\omega|^3}{\partial \nu_x} \, <\theta,\nu_x>      -2(n+2) \int_\omega <Du_\omega,Du'_\omega>  \>dx \\\\
& -(n+2)  \displaystyle\oint_{\partial \omega} |Du_\omega|^2 \, <\theta,\nu_x>  .
\end{array}
\end{equation}
By {  the divergence theorem, Problem \eqref{deltino}, and Problem \eqref{deltino2}, we get}

\begin{equation}
\label{der2}
\int_\omega <Du_\omega,Du'_\omega>  \>dx = \displaystyle\oint_{\partial \omega} u_\omega \frac{\partial u'_\omega}{\partial \nu_x}-\int_\omega  u_\omega \Delta u'_\omega \>dx =0.
\end{equation}
On the other hand
\begin{equation}
\label{der3}
\displaystyle\oint_{\partial \omega} |Du_\omega| <Du_\omega,Du'_\omega>= -\displaystyle\oint_{\partial \omega} |Du_\omega|^2 \frac{\partial u'_\omega}{\partial \nu_x}.
\end{equation}
Finally 
\[
\frac{\partial |Du_\omega|^3}{\partial \nu_x}= 3 |Du_\omega|^2 <D(|D u_\omega|), \nu_x>= -3 |Du_\omega| <D(|D u_\omega|), D u_\omega>.
\]
{  Bearing in mind} that for $\partial \omega =\{ x: u_\omega(x)=0\}$ it holds 
\[
(n-1) H = -{\rm div}\left(\displaystyle\frac{Du_\omega}{|Du_\omega|}\right)= \frac{1}{|Du_\omega|}+ \frac{1}{|Du_\omega|^2}<D(|D u_\omega|), D u_\omega>,
\]
\noindent we have
\begin{equation}
\label{der4}
\frac{\partial |Du_\omega|^3}{\partial \nu_x}= -3 |Du_\omega|  \left( (n-1) H - \frac{1}{|Du_\omega|} \right).
\end{equation}
Plugging \eqref{der2}, \eqref{der3}, \eqref{der4} {  into} \eqref{der1} we obtain \eqref{derivJ}.
\end{proof}

\begin{proof}[Proof of Theorem  \ref{S1}]
Let $\Omega$ be a solution to Serrin problem. By lemma \ref{funct}, $\Omega$ is a minimizer of $J$, then for every vector field $\theta \in C^2(\R^n;\R^n)$ we must have 
$$
dJ(\Omega,\theta)=0.
$$ 
Using \eqref{derivJ} together with $c= \dfrac{\partial u_\Omega}{\partial \nu_x}= -|Du_\Omega|$ on $\partial \Omega$ yields

\[
dJ(\Omega,\theta)=2c^2(n-1) \oint_{\partial\Omega} [ 1+nHc] <\theta,\nu_x> - 3nc^2 \oint_{\partial\Omega} \displaystyle\frac{\partial u'_\omega}{\partial \nu_x}.
\]
{ Equation \eqref{deltino2} gives} $$\oint_{\partial\Omega} \displaystyle\frac{\partial u'_\omega}{\partial \nu_x}=0,$$
and then

\[
dJ(\Omega,\theta)=2c^2(n-1) \oint_{\partial\Omega} [ 1+nHc] <\theta,\nu_x> =0, \quad \forall \>\> \theta \in C^2(\R^n;\R^n).
\]
Hence the mean curvature of $\partial \Omega$ is constant and Alexandrov theorem (see \cite{A1,A2}) implies that $\Omega$ is a ball. This concludes the proof.
\end{proof}

\subsubsection{Remark on the proof and { generalizations}.}
The proof via shape derivative is another nice example of proof which does not uses the maximum principle explicitly. Again however the constant sign of the solution $u$ is used. It is also interesting to notice that it uses Alexandrov theorem \cite{A1,A2} which in turn, at least in the original version, relies on the moving planes. Recently even a deeper connection between Alexandrov Theorem and Serrin problem has been exploited in \cite{CiMa, MP} on the wake of \cite{Ros}.

The shape derivative technique requires (following \cite{ChoHen}) somewhat more regularity than Weinberger's ones. It has been however successfully applied in other contexts for instance to obtain partial result toward the solution to the Schiffer conjecture (see \cite{ChaHen}).

\subsection{An integral approach via arithmetic-geometric mean inequality \cite{BNST}}
This idea stems from the need to extend Serrin overdetermined result to non uniformly elliptic operators of Hessian type. It is a fairly simple proof once we get acquainted with the notation used. We denote by $A=(a_{ij})$ a matrix in the space $\Sn$ of the real symmetric $n\times n$ matrices, and by
$\lambda_1,...,\lambda_n$ its eigenvalues, we define  the first and the second elementary symmetric function of  its eigenvalues as 
$$S_1(A) = Tr(A), \quad S_2(A)=S_2(\lambda_1,...,\lambda_n)=\sum_{1\leq i_1
<i_2\leq n } \lambda_{i_1} \lambda_{i_2}.$$ 
Note that $S_2(A)$ is just the sum of all $2 \times 2$ principal minors of $A$, and in dimension $2$ is nothing but
${\rm Det} A$.

Denoting by
$$
S_2^{ij}(A)
= \frac{\partial }{\partial a_{ij}}S_2(A),$$
Euler identity for homogeneous functions gives
$$
S_2(A) = \frac{1}{2} S_2^{ij}(A) a_{ij},
$$ 
here  we are adopting the Einstein summation convention for repeated indices.

Then the following inequality, known as Newton inequality, holds true in the class of  matrices whose trace is nonnegative

\begin{equation}\label{rel3}
(S_1(A))^2 \ge \frac{2n}{n-1} S_2 (A);
\end{equation}
equality  in \eqref{rel3} implies $\lambda_1= \lambda_2=...=\lambda_n $ (see \cite{HLP}).

Given a $C^2$ function $u$, the $k$-Hessian operators $S_k\left(D^2u\right)$ ($k=1,2$)
are defined as the $k$-th elementary symmetric function of $D^2u$.
Observe that with this notation  $$S_1(D^2u)=\Delta u.$$

\noindent A direct computation yields that $(S_2^{1j}(D^2u),\dots,S_2^{nj}(D^2u))$ is divergence free, i.e.
\begin{equation}
\label{div0}
\frac{\partial}{\partial x_i}S_2^{ij}=0;
\end{equation}
hence $S_2(D^2 u)$ can be written in the following divergence form
\begin{equation}\label{divk}
S_2(D^2 u) = \frac{1}{2} S_2^{ij}(D^2 u) u_{ij} =
 \frac{1}{2} (S_2^{ij}(D^2 u) u_j)_i,
\end{equation}
(from now on subscripts stand for partial differentiations).

Let $t$ be a regular value of $u$ and let $L=\{u \ge t\}$. 
If, with an abuse of notation, we denote by $H= -{\rm div}\left(\displaystyle\frac{Du}{|Du|}\right)$,  $(n-1)$ times the curvature  of the level set $\partial L$ at the point $x$, then 
$$-\Delta u=H |Du|-\frac{u_{ij}u_iu_j}{|Du|^2}.$$
{  This} means that the value of $\Delta u$ at any regular point (i.e. { a point with} non vanishing gradient) only involves derivates of $u$ along the direction of steepest descent and the mean curvature $H/(n-1)$ of the level surface through that point.

Finally the following pointwise identity holds (see \cite{Reilly})
\begin{equation}\label{H-S}
-H=\frac{S_2^{ij}(D^2 u)u_iu_j}{|Du|^{3}}.
\end{equation}

\begin{proof}[Proof of Theorem  \ref{S1}] {  First we observe that}

\begin{eqnarray}
&\displaystyle\int_\Omega|Du|^2 &=\int_\Omega |Du|^2 (-\Delta u)  \notag
\\
&&=2\int_\Omega u_{ij}u_iu_j -\oint_{\partial \Omega}|Du|^2<Du,\nu_x> \notag
\\
&&=2\int_\Omega\left[\Delta u|Du|^2+H|Du|^3\right] + c^3|\partial \Omega|. \notag
\end{eqnarray}
Then, using the equation in \eqref{delta} and the fact that $-c = \frac{|\Omega|}{|\partial \Omega|}$ we have
\begin{equation}\label{H}
\int_\Omega H|Du|^3=\frac{3}{2}\int_\Omega |Du|^2-\frac{c^2 }{2}|\Omega|.
\end{equation}
{ Plugging \eqref{uguale} and \eqref{mean} into \eqref{H} we obtain 

\begin{equation}\label{H_Omega}
 \int_\Omega H|Du|^3= \frac{(n-1)}{ (n+2)}c^2|\Omega|.
\end{equation}
Using that $u >0$ in $\Omega$, equations \eqref{divk}, \eqref{H-S}, \eqref{H_Omega}, and inequality \eqref{rel3} yield}
$$
\frac{(n-1)}{ (n+2)}c^2|\Omega|=\int_\Omega H|Du|^3=2 \int_\Omega u S_2(D^2 u) \le  \frac{n-1}{n}\int_\Omega u=\frac{(n-1)}{ (n+2)}c^2|\Omega|.
$$
This implies that equality holds true in \eqref{rel3} so
\begin{equation}
S_2(D^2 u)=\frac{n-1}{2n} \quad {\rm in } \>\> \Omega, 
\end{equation}
 and the Hessian matrix $D^2 u$ has all equal eigenvalues at every point of $\Omega$. This fact  implies that $D^2 u$ is {  a constant times} the identity matrix and the thesis follows.
\end{proof}

\subsubsection{Remark on the proof and generalization.}
Here is another example where besides the constant sign of the solution $u$ there is no shade of maximum principle. Basically the only ingredient of the proof is the geometric mean inequality. Once again the proof only needs the regularity required by the Poho\v{z}aev inequality i.e.: $u\in C^2(\Omega)\cap C^1(\bar\Omega)$. There is a deep connection between this proof and the Weinberger's proof, since the first one consists somehow in evaluating the integral over $\Omega$ of $\Delta u$ times the $P$-function. However no maximum principle on $P$ is established and everything is kept in integral form. Even if the proof was successfully applied to nonlinear operator of Hessian type, it turned out that the main advantage of this approach is that it does not use any pointwise argument. By means of this technique, stability theorem for Serrin problem like those in \cite{ABR} were improved in \cite{BNST2}. Moreover the technique is well designed when dealing with anisotropic overdetermined problem \cite{CiaSal}, where intrinsic asymmetry and lack of regularity advise against Serrin's and Weinberger's proofs.


\begin{thebibliography}{10}

\bibitem{ABR}
A.~Aftalion, J.~Busca, and W.~Reichel, \emph{Approximate radial symmetry for
  overdetermined boundary value problems}, Adv. Differential Equations
  \textbf{4} (1999), no.~6, 907--932.

\bibitem{A1}
A.D.~Alexandrov, \emph{Uniqueness theorems for surfaces in the large}, V. (Russian) Vestnik Leningrad. Univ . \textbf{13} (1958), 5--8.
  
  \bibitem{A2}
A.D.~Alexandrov, \emph{A characteristic property of the spheres}, Ann. Mat. Pura e Appl. \textbf{58} (1962), 303--354.




\bibitem{Be}
A.~Bennett, \emph{Symmetry in an ovedetermined fourth order elliptic boundary value problem}, SIAM J. Math. Anal. \textbf{17} (1986), 1354--1358.

\bibitem{BNST}
B.~Brandolini, C.~Nitsch, P.~Salani, and C.~Trombetti, \emph{Serrin type
  overdetermined problems: an alternative proof}, Arch. Rat. Mech. Anal. \textbf{190} (2008), 267--280.
  
  \bibitem{BNST2}
B.~Brandolini, C.~Nitsch, P.~Salani, and C.~Trombetti, \emph{On the stability of the Serrin problem}, J. Diff. Equations \textbf{245} no.6 (2008), 1566--1583.
  
  \bibitem{BuKaw}
  G.~Buttazzo and B.~Kawohl, \emph{Overdetermined boundary value problems for the $\infty$-Laplacian}, Int. Math. Res. Not. IMRN \textbf{2} (2011), 237--247.
  
  \bibitem{ChaHen}
T.~Chatelain and A.~Henrot., \emph{Some results about Schiffer's conjecture}, Inverse Problems
  \textbf{15} (1999),  647--658.

 \bibitem{ChoHen}
M.~Choulli and A.~Henrot, \emph{Use of the domain derivative to prove symmetry results in partial differential equations}, Math. Nachr.
  \textbf{192} (1998), 91--103.
  
  \bibitem{CiaSal}
A.~Cianchi  and P.~Salani, \emph{Overdetermined anisotropic elliptic problems}, Math. Ann.
  \textbf{345} (2009), 859--881.
  
   \bibitem{CiMa}
G.~Ciraolo  and F.~Maggi, \emph{On the shape of compact hypersurfaces with almost constant mean curvature}, Comm. Pure Appl. Math.
  \textbf{70} (2017), 665--716.
  
  

\bibitem{CraFra}
G.~Crasta and I.~Fragal\'a, \emph{A symmetry problems for the infinity Laplacian}, Int. Math. Res. Not. IMRN \textbf{18} (2015), 
  8411--8436.



\bibitem{FaKa} 
A.~Farina and B.~Kawohl, \emph{Remarks on an overdetermined boundary value problem}, Calc. Var. Partial Differential Equations \textbf{31} (2008), no. 3, 351--357.

\bibitem{FGK}
I.~Fragal\'a, F.~Gazzola and B.~Kawohl, \emph{Overdetermined problems with possibly degenerate ellipticity, a geometric approach}, Math. Z.
  \textbf{254} (2006), 117--132.

\bibitem{GL}
N.~Garofalo and J.L.~Lewis, \emph{A symmetry result related to some overdetermined boundary value problems}, American J. of Math.
  \textbf{111} (1999),  no.~1, 9--33.
  
  \bibitem{GNN}
B.~Gidas, Wei~Ming Ni, and L.~Nirenberg, \emph{Symmetry and related properties
  via the maximum principle}, Comm. Math. Phys. \textbf{68} (1979), no.~3,
  209--243.



\bibitem{had} 
J.~ Hadamard, \emph{M\'emoire sur le probl\'eme d'analyse relatif \`a l'\'equilibre des plaques \'elastiques encastr\'ees}, (1908) oeuvres de J. Hadamard, CNRS, Paris, 1968. 



\bibitem{HLP}
G.H.~Hardy, J.E.~Littlewood, and G.~P{\'o}lya, \emph{Inequalities}, Cambridge
  Mathematical Library, Cambridge University Press, Cambridge, 1988, Reprint of
  the 1952 edition.


\bibitem{HenPi} 
A.~Henrot and M.~Pierre, \emph{Variation et optimisation de formes. Une analyse g\'eom\'etrique}, Math\'ematiques \& Applications (Berlin), 48. Springer, Berlin, 2005.

\bibitem{MP}
R.~Magnanini  and  G.~Poggesi \emph{On the stability for Alexandrov's soap bubble theorem}, J. Anal Math, to appear.

\bibitem{PaSc}
L.E.~Payne  and  P.W.~Schaefer \emph{Duality theorems in some overdetermined problems}, Math. Methods in the Appl. Sciences,
   \textbf{11} (1989), 805--819.
   
   

\bibitem{Poh65}
S.I.~Poho{\v{z}}aev, \emph{On the eigenfunctions of the equation {$\Delta
  u+\lambda f(u)=0$}}, Dokl. Akad. Nauk SSSR \textbf{165} (1965), 36--39.
  
  \bibitem{PSZ}
G.~P\'{o}lya and G.~Szeg\"{o}, \emph{Isoperimetric Inequalities in
Mathematical Physics}, Ann.\ of Math.\ Stud.\ 27, Princeton
University Press, Princeton, 1951.



\bibitem{Reilly}
R.C.~Reilly, \emph{On the {H}essian of a function and the curvatures of its
  graph}, Michigan Math. J. \textbf{20} (1973), 373--383.
  

\bibitem{Ros} 
A.~Ros, \emph{Compact hypersurfaces with constant higher order mean curvatures}, Rev.
Mat. Iberoamericana, \textrm{3} (1987), 447--453.


\bibitem{Serrin}
J.~Serrin, \emph{A symmetry problem in potential theory}, Arch. Rational Mech.
  Anal. \textbf{43} (1971), 304--318.

\bibitem{Struwe}
M.~Struwe, \emph{Variational methods}, third ed., Ergebnisse der Mathematik und
  ihrer Grenzgebiete. 3. Folge. A Series of Modern Surveys in Mathematics
  [Results in Mathematics and Related Areas. 3rd Series. A Series of Modern
  Surveys in Mathematics], vol.~34, Springer-Verlag, Berlin, 2000, Applications
  to nonlinear partial differential equations and Hamiltonian systems.

%



\bibitem{Wein}
H.F.~Weinberger, \emph{Remark on the preceding paper of {S}errin}, Arch.
  Rational Mech. Anal. \textbf{43} (1971), 319--320.

\end{thebibliography}
\end{document}